\documentclass[12pt,reqno]{amsart}
\usepackage{amscd,amssymb,amsmath,amsthm}
\usepackage[arrow,matrix]{xy}
\usepackage{graphicx}
\usepackage{cite}
\usepackage{geometry}
\newtheorem{thm}{Theorem}
\newtheorem{defn}{Definition}
\newtheorem{lem}{Lemma}
\newtheorem{pro}{Proposition}

\newtheorem{cor}{Corollary}

\newtheorem{exa}{Example}

\numberwithin{equation}{section}

\begin{document}
\title[Subgroups and normal subgroups]
{Invariance property on group representations of the Cayley tree and
its applications}

\author{U. A. Rozikov, F. H. Haydarov}
\address{U. A. Rozikov \\ Institute of Mathematics,\\
Tashkent, Uzbekistan.}
 \email {rozikovu@yandex.ru.}

 \address{F.H.Haydarov\\ National University of Uzbekistan,
Tashkent, Uzbekistan.}
 \email {haydarov\_ imc@mail.ru.}

\begin{abstract}

In the present paper, we give a certain condition on subgroups
of the group representation of the Cayley tree such that an
invariance property holds. Except for the given condition, we show
that the invariance property does not hold.

\end{abstract}

\maketitle



\section{Introduction}

   There are many open problems of group theory which arise in studying of
   problems of natural sciences such as physics, mechanics, biology and so on.
  For instance, a configuration of a physical system on tree can be considered as a function defined on
  the set of vertices of the tree (see \cite{nr},\cite{7},\cite{vvnr}). On the set of configurations of a model one
  defines a Gibbs measure. To introduce periodic and weakly periodic
  Gibbs measures one needs subgroups of the group representation of the tree (\cite{5}).
   The invariance property of subgroups of group representation of Cayley tree is sufficient
   to study periodic or weakly periodic Gibbs measures on a Cayley
   tree.

   For any normal subgroup of finite index of group representation of Cayley
   tree, the invariance property holds, that's why the theory of
   periodic and weakly periodic Gibbs measures corresponding to normal subgroups has been well developed.
    On the other hand side, up to now for any subgroup (not normal) of the group
   representation of the tree, the invariance property has not
   been proved, so there is no any result weakly periodic and
   periodic Gibbs measures corresponding to subgroups (not normal)
   of group representation of Cayley tree (see detail in \cite{5}).

   Thus, if the invariance property holds then we can study periodic and
weakly periodic Gibbs measures corresponding to an arbitrary
subgroup of finite index of group representation of Cayley tree.

In this paper we find a certain condition on subgroups (not normal)
of the group representation of Cayley tree under which the
invariance property holds.
\section{Invariance property and its applications}
\subsection{Subgroups of odd index.}
Let $G_{k}$ be a free product of $k+1$ cyclic groups of the second
order with generators $a_{1},a_{2},...a_{k+1},$ respectively.

\begin{thm}\label{th0}\emph{\cite{5}}
\item 1. The group $G_{k}$ does not have normal subgroups of odd
index $(\neq 1)$.
\item 2. The group $G_{k}$ has normal subgroups of arbitrary even index.
\end{thm}

By Theorem $\ref{th0}$ we can conclude that any subgroup of odd
index of the group $G_k$ is not normal subgroup.
Up to now, the problem of existence of subgroup of any odd index
of the group $G_k$ has been being an open problem. Here, we
give subgroups of index $2s+1$ of the group
$G_k$, where $s\in \mathbb{N}$.

 Let $N_k=\{1,2,...,k+1\}$ and $A_0\subset N_k, 0\leq|A_{0}|\leq k-1.$ $(A_{1}, A_{2})$ be a partition of the set
 $N_{k}\backslash A_{0}$. Put $m_{j}$ be a minimal element of
$A_{j},\,\ j\in\{1,2\}.$ Then we consider the function
$u_{A_{1}A_{2}}:\{a_{1},a_{2},...,a_{k+1}\}\rightarrow \{e,
a_{1}...,a_{k+1}\}$ (where $e$ is identity element) given by
\begin{equation}\label{func1}u_{A_{1}A_{2}}(x)=\left\{\begin{array}{ll}
e, \ \ \mbox{if} \ \ x=a_{i},\ i\in N_{k}\setminus (A_{1}\cup A_{2})\\[3mm]
a_{m_{j}}, \ \mbox{if}\ \ x=a_{i},\ i\in A_{j},\ j=1,2.
\\ \end{array}\right.\end{equation}
For $1\leq q\leq s$ define $\gamma_s: <e, a_{m_{1}},
a_{m_{2}}> \ \rightarrow <e, a_{m_{1}}, a_{m_{2}}>$ by the formula\\
\begin{equation}\label{func2}\gamma_s(x)=\left\{\begin{array}{lll}
e  \ \mbox{if} \ \ x=e\\[2mm]
a_{m_{1}}a_{m_{2}}a_{m_{1}}...a_{m_{j}} \ \mbox{if} \ \
x\in\{\underbrace{a_{m_{1}}a_{m_{2}}a_{m_{1}}...a_{m_{j}}}_{q},\
\underbrace{a_{m_{2}}a_{m_{1}}a_{m_{2}}...a_{m_{3-j}}}_{2s+1-q}\}\\[2.5mm]
a_{m_{2}}a_{m_{1}}a_{m_{2}}...a_{m_{j}} \ \mbox{if} \ \
x\in\{\underbrace{a_{m_{2}}a_{m_{1}}a_{m_{2}}...a_{m_{j}}}_{q},\
\underbrace{a_{m_{1}}a_{m_{2}}a_{m_{1}}...a_{m_{3-j}}}_{2s+1-q}\}\\[2.5mm]
\gamma_s(a_{m_{j}}...\gamma_s(\underbrace{a_{m_{j}}a_{m_{3-j}}...a_{m_{3-j}}}_{2s}))
 \ \mbox{if} \ \ x=a_{m_{j}}a_{m_{3-j}}...a_{m_{3-j}},\ l(x)>2s \\[2.5mm]
\gamma_s(a_{m_{j}}...\gamma_s(\underbrace{a_{m_{3-j}}a_{m_{j}}...a_{m_{j}}}_{2s}))
 \ \mbox{if} \ \ x=a_{m_{j}}a_{m_{3-j}}a_{m_{j}}...a_{m_{j}},\ l(x)>2s,
 \end{array}\right.\end{equation}
where $l(x)$ is the length of $x$.

Denote
\begin{equation}\label{func3}\Im^{s}_{A_{1}A_{2}}(G_{k})=\{x\in G_{k}|\
\gamma_s(u_{A_{1}A_{2}}(x))=e\}\end{equation}

\begin{lem}\label{lem1} Let $(A_{1}, \  A_{2})$
be a partition of the set $N_{k}\backslash A_{0}, \,\
0\leq|A_{0}|\leq k-1.$ Then $x\in \Im^{s}_{A_{1}A_{2}}(G_{k})$ if
and only if the number $l(u_{A_{1}A_{2}}(x))$ is divisible by
$2s+1$.
\end{lem}
\begin{proof} Let $x=a_{i_{1}}a_{i_{2}}...a_{i_{n}}\in G_{k}$
and $l(u_{A_{1}A_{2}}(x))$ be odd (the case even is similar).
 Then we can write $u_{A_{1}A_{2}}(x)=a_{m_{i}}a_{m_{3-i}} ...
 a_{m_{i}}a_{m_{3-i}}a_{m_{i}}.$ Then $\gamma_s(u_{A_{1}A_{2}}(x))$ is
 equal to\\
 $$\gamma_s(a_{m_{i}}a_{m_{3-i}} ... a_{m_{i}}a_{m_{3-i}}a_{m_{i}})=\gamma_s(a_{m_{i}}a_{m_{3-i}} ...
 a_{m_{i}}\gamma_s(\underbrace{a_{m_{3-j}}a_{m_{j}}...a_{m_{j}}}_{2s}))=$$
 $$=\gamma_s(\underbrace{a_{m_{i}}a_{m_{3-i}} ...
 a_{m_{3-i}}a_{m_{i}}}_{l(u_{A_{1}A_{2}}(x))-2s-1})=...$$
 We continue this process until the length of the word
 $\gamma_s(u_{A_{1}A_{2}}(x))$ will be less than $2s+1$. From
 $\gamma_s(u_{A_{1}A_{2}}(x))=e$ we deduce that
 $l(u_{A_{1}A_{2}}(x))$ is divisible by $2s+1$.

 Conversely, if $l(u_{A_{1}A_{2}}(x))$ is divisible by $2s+1$ then $x$ is generated by the
 elements $\underbrace{a_{m_{1}}a_{m_{2}}a_{m_{1}}...a_{m_{1}}}_{2s+1}$ and
$\underbrace{a_{m_{2}}a_{m_{1}}a_{m_{2}}...a_{m_{2}}}_{2s+1}.$
Since
$\gamma_s(\underbrace{a_{m_{1}}a_{m_{2}}a_{m_{1}}...a_{m_{1}}}_{2s+1})=\gamma_s(\underbrace{a_{m_{2}}a_{m_{1}}a_{m_{2}}...a_{m_{2}}}_{2s+1})=e,$
 we get $x\in \Im^{s}_{A_{1}A_{2}}(G_{k}).$
\end{proof}
In spite of giving a full description of subgroups of the group
$G_{k}$ is too hard, there are some papers which devoted to the
description of normal subgroups of $G_k.$(see \cite{hay},
\cite{4}) The next proposition gives us a full description of all
(not normal) subgroups of index three.
\begin{pro}\label{pr2} For the group $G_{k}$ the following
equality holds
$$\{K|\ K\ is\ a\ subgroup\ of\ G_{k}\ of \ index \ 3\}=\ \ \ \ \ \ \ \ \ \ \ \ \ \ \ \
\ \ \ \ \ \ \ \ \ \ \ \ \ \ \ \ \ \ \ \ $$
$$\ \ \ \ \ \ \ \ \ \ \ \ \ \ \ \ \ \ \ \ =\{\Im^{1}_{A_{1}A_{2}}(G_{k})|\ A_{1},\ A_{2} \ is\ a\ partition\ of\ N_{k}\setminus A_{0}\}.$$
\end{pro}
\begin{proof} Let $K$ be a subgroup of the group
$G_{k}$ with $|G_{k}:K|=3$. Then there exist $p,q\in N_{k}$ such
that $|\{K,\ a_{p}K,\ a_{q}K\}|=3$. Put\\
$$A_{0}=\{i\in N_{k}|\ a_{i}\in K\},
 \ \ \ \ A_{1}=\{i\in N_{k}|\ a_{i}\in a_{p}K\},
 \ \ \ \ A_{2}=\{i\in N_{k}|\ a_{i}\in a_{q}K\}$$\\
From $|G_{k}:K|=3$ we conclude that $\{A_{1}, A_{2}\}$ is a
partition of $N_{k}\setminus A_{0}.$ Let $m_{i}$ be a minimal
element of $A_{i},\ i=1,2$. If we show that
$\Im^{1}_{A_{1}A_{2}}(G_{k})$ is a subgroup of $G_{k}$
(corresponding to $K$) then it'll be given one to one
correspondence between given sets. For
$x=a_{i_{1}}a_{i_{2}}...a_{i_{n}}\in \Im^{1}_{A_{1}A_{2}}(G_{k}),\
y=a_{j_{1}}a_{j_{2}}...a_{j_{m}}\in \Im^{1}_{A_{1}A_{2}}(G_{k})$
it is sufficient to show that
$xy^{-1}\in\Im^{1}_{A_{1}A_{2}}(G_{k}).$ Let\
$u_{A_{1}A_{2}}(x)=a_{m_{i}}a_{m_{3-i}}....a_{m_{s}},\
u_{A_{1}A_{2}}(y)=a_{m_{j}}a_{m_{3-j}}....a_{m_{t}}$, where
$(i,j,s,t)\in \{1,2\}^{4}$. Since $u_{A_{1}A_{2}}$ is a
homomorphism we have\\
$$l\left(u_{A_{1}A_{2}}(xy^{-1})\right)=\left\{\begin{array}{ll}
|l\left(u_{A_{1}A_{2}}(x))-l(u_{A_{1}A_{2}}(y^{-1})\right)|, \
 \ \mbox{if} \ \ s=t\\[2.5mm]
l\left(u_{A_{1}A_{2}}(x))+l(u_{A_{1}A_{2}}(y^{-1})\right),
\ \ \mbox{if} \ \ s\neq t.\\
\end{array}\right.$$\\

By Lemma \ref{lem1} we see that $l\left(u_{A_{1}A_{2}}(x)\right)$
is divisible by 3 and
$l\left(u_{A_{1}A_{2}}(y^{-1})\right)=l\left(\left(u_{A_{1}
A_{2}}(y)\right)^{-1}\right)$ is also divisible by 3, so
$l\left(u_{A_{1}A_{2}}(xy^{-1})\right)$ is a multiple of 3, which
shows that $xy^{-1}\in \Im^{1}_{A_{1}A_{2}}(G_{k})$.
\end{proof}

\subsection{The invariance property}
Now, we study the invariance property for above defined subgroups
of the group $G_k$.

For any element $x$ of $G_k$, denote by $x_{\downarrow}$ on element which
satisfies the following condition: $x^{-1}\cdot
x_{\downarrow}\in\{a_i\ | \ i\in N_k \}$. The set of all direct
successors of $x$ is defined by $S(x)=\{y\in G_k\ |
y_{\downarrow}=x\}.$

\begin{pro}\label{l1}(\textbf{Invariance property}) For $A_{i}=\{m_i\}$ and $x,y\in G_k$ if  $\gamma_s(u_{A_{1}A_{2}}(x))=\gamma_s(u_{A_{1}A_{2}}(y)),$
$\gamma(u_{A_{1}A_{2}}(x_{\downarrow}))=\gamma_s(u_{A_{1}A_{2}}(y_{\downarrow}))$
then
$$\langle\gamma_s(u_{A_{1}A_{2}}(xa_{i}))\ |\ xa_i\in S(x)\rangle=\langle\gamma_s(u_{A_{1}A_{2}}(ya_{i}))\ |\ ya_i\in S(y)\rangle,$$
where $\langle...\rangle$ stands for ordered $k$-tuples. \end{pro}
\begin{proof} If $i\notin A_{1}\cup A_2$ then it is easy to check
that $u_{A_{1}A_{2}}(xa_{i})=u_{A_{1}A_{2}}(x)$. Consequently,
$$\gamma_s(u_{A_{1}A_{2}}(xa_{i}))=\gamma_s(u_{A_{1}A_{2}}(x))=\gamma_s(u_{A_{1}A_{2}}(y))=\gamma_s(u_{A_{1}A_{2}}(ya_{i})).$$
For the case $i\in A_{1}\cup A_2$, we show
$\gamma_s(u_{A_{1}A_{2}}(xa_{i}))\neq \gamma_s(u_{A_{1}A_{2}}(x))$
i.e., $\gamma_s(u_{A_{1}A_{2}}(xa_{m_i}))\neq
\gamma_s(u_{A_{1}A_{2}}(x))$. If
$l(u_{A_{1}A_{2}}(xa_{m_{i}}))\geq 4s+2$, then we denote $x^{(1)}$
as a new element which occurs by removing the last $4s+2$ letters
of $u_{A_{1}A_{2}}(x)$. Clearly,
$\gamma_s(u_{A_{1}A_{2}}(xa_{m_i}))=\gamma_s(u_{A_{1}A_{2}}(x^{(1)}a_{m_i}))$.
If $l(x^{(1)}a_{m_i})\geq 4s+2$ then we'll continue the above
process and get another element $x^{(2)}$, and chosen element
satisfies the following
$$\gamma_s(u_{A_{1}A_{2}}(x^{(1)}a_{m_i}))=\gamma_s(u_{A_{1}A_{2}}(x^{(2)}a_{m_i})).$$
This process carries on when we get
$l(x^{(n)}a_{m_i})<4s+2$. Hence, we get
$$\gamma_s(u_{A_{1}A_{2}}(x^{(n)}a_{m_i}))=\gamma_s(u_{A_{1}A_{2}}(x^{(n-1)}a_{m_i}))=...
=\gamma_s(u_{A_{1}A_{2}}(xa_{m_i})),$$ and
$\gamma_s(u_{A_{1}A_{2}}(x^{(n)}))=\gamma_s(u_{A_{1}A_{2}}(x))$.
Similarly, we get
$\gamma_s(u_{A_{1}A_{2}}(y^{(r)}))=\gamma_s(u_{A_{1}A_{2}}(y))$,
where $l(y^{(r)}a_{m_i})<4s+2$, $r\in \mathbb{N}$.

 If $x^{(n)}=y^{(r)}$ then from $x_{\downarrow}=y_{\downarrow}$
we have
$$\langle\gamma_s(u_{A_{1}A_{2}}(xa_{i}))\ |\ xa_i\in
S(x)\rangle=\langle\gamma_s(u_{A_{1}A_{2}}(ya_{i}))\ |\ ya_i\in S(y)\rangle.$$

Let $x^{(n)}\neq y^{(r)}$. In this case, from defining of the
function $\gamma_s$ one gets $l\left((x^{(n)})^{-1}
y^{(r)}\right)=2s+1.$ Hence, the last letters of words $x^{(n)}$
and $y^{(r)}$ coincide and we can conclude that
$\gamma_s(x^{(n)}a_{m_j})=\gamma_s(y^{(r)}a_{m_{3-j}}), \
j\in\{1,2\}$. From the last equality and conditions of the
proposition, the following equality holds:
$\langle\gamma_s(u_{A_{1}A_{2}}(xa_{i}))\ |\ xa_i\in
S(x)\rangle=\langle\gamma_s(u_{A_{1}A_{2}}(ya_{i}))\ |\ ya_i\in S(y)\rangle.$
\end{proof}

Generally speaking, Proposition \ref{l1} does not hold for all cases. Now, we give one example
which Proposition \ref{l1} doesn't hold.
\begin{exa}\label{ex+} Let $\mathcal{K}:=\Im^{1}_{\{1, 3\}\{2\}}(G_{2})$ and we choose $x=a_2a_1a_2, \ y=a_1a_3$.
 Then we have
$\gamma_1(u_{\{1, 3\}\{2\}}(x))=\gamma_1(u_{\{1, 3\}\{2\}}(y)),$
$\gamma(u_{\{1, 3\}\{2\}}(x_{\downarrow}))=\gamma_1(u_{\{1, 3\}\{2\}}(y_{\downarrow}))$.
On the other handside,
$$\langle\gamma_1(u_{\{1, 3\}\{2\}}(xa_{i}))\ |\ xa_i\in S(x)\rangle=\langle a_2, a_2\rangle,$$
$$\langle\gamma_1(u_{\{1, 3\}\{2\}}(ya_{i}))\ |\ ya_i\in S(y)\rangle=\langle a_1, a_2\rangle.$$
Namely, the invarience property does not hold: $\langle a_2, a_2\rangle\neq \langle a_1, a_2\rangle$.
\end{exa}

Now we introduce an equivalence relation on $G_k.$ Namely,
elements $x$ and $y$ are called equivalent $(x\sim y)$ if it
satisfies this condition:
$\gamma(u_{A_{1}A_{2}}(x)=\gamma(u_{A_{1}A_{2}}(y).$

Let $K_0$ be a subgroup of $G_k$ and $G_{k}/ K_0=\{K_0, K_1,
K_2,..., K_{2s}\},$ where $K_i, i\in\{0,1,...,2s\}$ are cosets.

 For $i\in\{1,2,...,2s\}$ we denote
\begin{equation}\label{notation1}S_{1}(x)=S(x)\cup \{x_{\downarrow}\},\ \
q_{i}(x)=|S_{1}(x)\cap K_{i}|,\ \
Q(x)=(q_{0}(x),q_{1}(x),q_{2}(x),..., q_{2s}(x)),\end{equation}
and
\begin{equation}\label{notation2} q_{i}(K_0)=q_{i}(e)=|\{j : a_{j}\in K_{i}\}|, \,\
Q(K)=(q_{0}(K_0),...,q_{n-1}(K_0)).\end{equation}

\begin{cor}\label{l2} Let $|A_{j}|=1,$ for $j=1,2$. If $x\sim y$, then
$q_{i}(x)=q_{i}(y)$ for $i=0,...,2s$.
\end{cor}
\begin{proof} If $p\notin A_1\cup A_2,$ then from $x\sim y$, one gets $xa_{i}\sim
x\sim y\sim ya_{i}$.
 That's why it suuficies to check the case $p\in A_1\cup A_2.$ Let
 $p\in A_1$ (the case $p\in A_2$ is similar) and $xa_{p}\sim ya_{p}$
 then by virtue of Proposition $\ref{l1}$, the relation $xa_{q}\sim ya_{q}, q\in
 A_2$ holds i.e., $q_{i}(x)=q_{i}(y)$. Analogously, if $xa_{p}\nsim ya_{p}$ then
 from the second lemma we get $xa_{p}\sim ya_{q}$ and $xa_{q}\sim
 ya_{p}.$ Also, in that case we have $q_{i}(x)=q_{i}(y).$ \end{proof}

\begin{thm}\label{thm+} Let $K_0\in\Im^s_{A_{1}A_{2}}(G_{k}), |A_{j}|=1$ for $j=1,2$.
For any $x\in G_{k},$ there exists a permutation $\pi_x$ of the
coordinates of the vector $Q(K_{0})$ such that
$$\pi_xQ(K_{0})=Q(x).$$
\end{thm}
\begin{proof} Clearly,
$S_{1}(x)=xS_{1}(e)=\{xa_{1},xa_{2},...,xa_{k+1}\}.$ By Proposition
$\ref{l1}$ and Corollary $\ref{l2}$, for any $i=\overline{0,2s}$
there exists $j(i)\in\{0,1,2,..., 2s\}$ such that
$$q_{i}(K_{0})=|\{j: \ a_{j}\in K_{i}\}|=|\{xa_{m}: xa_{m}\in
K_{j(i)}\}|=q_{j(i)}(x).$$ As a result, we choose a permutation as
$\pi_x(i)=j(i)$.\end{proof}

\subsection{Application of the invariance property} A Cayley tree (Bethe lattice) $\Gamma^k$ of order
$k\geq 1$ is an infinite homogeneous tree, i.e., a graph without
cycles, such that exactly $k+1$ edges originate from each vertex.
Let $\Gamma^k=(V,L)$ where $V$ is the set of vertices and $L$ that
of edges (arcs). It is known that there exists a one to one
correspondence between the set of vertices $V$ of the Cayley tree
$\Gamma^{k}$ and the group $G_{k}$ (see \cite{5}).

We consider models where the spin takes values in the set
$\Phi:=\{-1,1\}$, and is assigned to the vertexes of the tree. For
$A\subset V$ a configuration $\sigma_A$ on $A$ is an arbitrary
function $\sigma_A: A\to \{-1,1\}$. Denote $\Omega_A=\{-1,1\}^A$
the set of all configurations on $A$.

 The (formal) Hamiltonian of ferromagnetic Ising model is:
 \begin{equation}\label{ising} H(\sigma)=-J\sum_{\langle x,y\rangle\in L} \sigma(x)\sigma(y),
\end{equation} where $J >0$ is a coupling constant and $\langle x,y\rangle$ stands for nearest neighbor vertices.
 For a fixed $x^0\in V$, called the root, we set
 \begin{equation}\label{ball}W_n=\{x\in V| \ d(x^0,x)=n\}, \ \ \
 V_n=\bigcup\limits_{m=0}^n W_m,\end{equation} where $d(x,y)$ (distance from $x$ to $y$ on the Cayley tree)
  - is the number of edges of the shortest path from $x$ to $y$.

Define a finite-dimensional distribution of a probability measure
$\mu$ in the volume $V_n$ as
\begin{equation}\label{2.6 eq}
\mu_n(\sigma_n)=Z_n^{-1}\exp\left\{-\beta H_n(\sigma_n)+\sum_{x\in
W_n}h_x \sigma(x)\right\},
\end{equation}
where $\beta=\frac{1}{T}, T>0$-temperature, $Z_{n}^{-1}$ is the
normalizing factor, $\{h_x\in R, x\in V\}$ is a collection of real
numbers and
$$H_n(\sigma_n)=-J\sum_{\langle x,y\rangle\in L}\sigma(x)\sigma(y).$$
We say that the probability distributions \label{2.6 eq} are
compatible if for all $n\geq 1$ and $\sigma_{n-1}\in
\Phi^{V_{n-1}}:$
\begin{equation}\label{2.7 eq} \sum_{\omega_n\in
\Phi^{W_{n}}}\mu_n (\sigma_{n-1}\vee
\omega_n)=\mu_{n-1}(\sigma_{n-1}).
\end{equation}
Here $\sigma_{n-1}\vee \omega_n$ is the concatenation of the
configurations. In this case, according to the Kolmogorov
extension theorem (\cite{ash}), there exists a unique measure
$\mu$ on $\Phi^{V_n},$
$$\mu(\{\sigma | _{V_n}=\sigma_n\})=\mu_n(\sigma_n).$$
Such a measure is called a splitting Gibbs measure corresponding
to the Hamiltonian (\ref{ising}) and function $h_x, x\in V.$

The following statement describes conditions on $h_x$ guaranteeing
compatibility of $\mu_n(\sigma_n).$
\begin{thm}\label{th2.1}\emph{\cite{5}}
Probability distributions $\mu_n(\sigma_n), n=1,2,...,$ in
(\ref{2.6 eq}) are compatible iff for any $x\in V$ the following
equation holds: \begin{equation}\label{2.8 eq} h_x=\sum_{y\in
S(x)}f(h_y, \theta).\end{equation} Here, $\theta= \tanh(J\beta),
f(h, \theta)=arctanh (\theta\tanh h)$ and $S(x)$ is the set of
direct successors of $x$ on Cayley tree of order $k$.\end{thm}
Since the set of vertices $V$ has the group representation $G_k.$
Without loss of generality we identify $V$ with $G_k,$ i.e., we
sometimes replace $V$ with $G_k.$
\begin{defn}\label{defn2.1} Let $K$ be a subgroup of $G_k,$ $k\geq
1.$ We say that a function $h=\{h_x\in R: x\in G_k\}$ is $K$-
periodic if $h_{yx}=h_x$ for all $x\in G_k$ and $y\in K.$ A $G_k$-
periodic function $h$ is called translation-invariant.
\end{defn}
A Gibbs measure is called $K$-\emph{periodic} if it corresponds to
$K$-periodic function $h$. Let $G_k: K=\{K_1,..., K_r\}$ be a
family of cosets, $K$ is a subgroup of index $r\in \mathbb{N}.$
\begin{defn}\label{defn 2.3} A set of quantities $h=\{h_x, x\in
G_k\}$ is called $K$-weakly periodic, if $h_x=h_{ij},$ for any
$x\in K_i, \ x_{\downarrow}\in K_j$.
\end{defn}
We note that the weakly periodic set of $h$ coincides with an
ordinary periodic one (see Definition \ref{defn2.1}) if the
quantity $h_x$ is independent of $x_{\downarrow}$.

A Gibbs measure $\mu$ is said to be $K$-\emph{weakly periodic} if
it corresponds to the $K$-weakly periodic set of $h$.

$K$-weakly periodic Gibbs measure of some models only has been
studying for the case $K$ is a normal subgroup of the group $G_k$.
Now we consider $K$-weakly periodic Gibbs measures of Ising model
on a Cayley tree of order $k$ for the case $K$ is not normal
subgroup of the group $G_k$.

Let $A_0=\{3,..., k+1\},$
$A_{s}=\{s\},\ s\in \{1,2\}$, i.e., $m_i=i,\ i\in\{1,2\}$. Now, we
consider functions $u_{\{1\}\{2\}}:\{a_{1},a_{2},...,
a_{k+1}\}\rightarrow \{e, a_{1}, a_{2}\}$ (defined in
(\ref{func1})) and $\gamma: <e, a_{1}, a_{2}> \ \rightarrow \{e,
a_{1}, a_{2}\}$ (defined in (\ref{func2})):
\begin{equation}\label{new1}u_{\{1\},\{2\}}(x)=\left\{\begin{array}{ll}
e, \ \ \mbox{if} \ \ x=a_{i}, \ i=\overline{3,k+1} \\[2.5mm]
a_{i}, \ \mbox{if}\ \ x=a_{i},\ i=\overline{1,2}\ ,
\end{array}\right.\end{equation}
\begin{equation}\label{new2} \gamma(x)=\left\{\begin{array}{lll}
e  \ \mbox{if} \ \ x=e\\[2mm]
a_{1} \ \mbox{if} \ \ x\in\{a_{1},\ a_{2}a_{1}\}\\[2.5mm]
a_{2} \ \mbox{if} \ \ x\in\{a_{2},\
a_{1}a_{2}\}\\[2.5mm]
\gamma\left(a_{i}a_{3-i}...\gamma(a_{i}a_{3-i})\right)
 \ \mbox{if} \ \ x=a_{i}a_{3-i}...a_{3-i},\ l(x)\geq 3,\
 i=\overline{1,2}\\[2.5mm]
\gamma\left(a_{i}a_{3-i}...\gamma(a_{3-i}a_{i})\right)
 \ \mbox{if} \ \ x=a_{i}a_{3-i}...a_{i},\ l(x)\geq 3,\
 i=\overline{1,2}.\\
 \end{array}\right.\end{equation}

Let $K_{0}^{\ast}:=\Im^{1}_{\{1\}\{2\}}(G_{k})$ (defined in
\ref{func3}). Then
 $$K_{0}^{\ast}=\{x\in G_{k}|\
\gamma(u_{\{1\}\{2\}}(x))=e\}.$$

By using $K_0^{\ast}$ is a subgroup of index 3 of the group
$G_{k}$ we define a family of cosets:
$$G_{2}/K^{\ast}_{0}=\{K^{\ast}_{0}, K^{\ast}_{1},
K^{\ast}_{2}\},$$ where
$$K_{1}^{\ast}=\{x\in G_{2}|\
\gamma(u_{\{1\}\{2\}}(x))=a_{1}\}, \ \ \ K_{2}^{\ast}=\{x\in
G_{2}|\ \gamma(u_{\{1\}\{2\}}(x))=a_{2}\}.$$ From notations in
(\ref{notation1}), we have
$$q_{0}(K_{0}^{\ast})=|\{a_1, a_2,..., a_{k+1}\}\cap K_{0}^{\ast}|=|\{a_{3},a_4,..., a_{k+1}\}|=k-1,$$
$$q_{1}(K_{0}^{\ast})=|\{a_1, a_2,..., a_{k+1}\}\cap K_{1}^{\ast}|=|\{a_{1}\}|=1,$$
$$q_{2}(K_{0}^{\ast})=|\{a_1, a_2,..., a_{k+1}\}\cap K_{2}^{\ast}|=|\{a_{2}\}|=1.$$
Thus, from notations in (\ref{notation2}) one gets
$Q(K_{0}^{\ast})=\{k-1,1,1\}$.

Assume now $x\in K_{1}^{\ast}$ (case $x\in K_{2}^{\ast}$ is
similar), namely, $\gamma(u_{\{1\}\{2\}}(x))=a_{1}$. Then
$\gamma(u_{\{1\}\{2\}}(xa_1))\neq\gamma(u_{\{1\}\{2\}}(xa_2))$ and
$\{\gamma(u_{\{1\}\{2\}}(xa_1)),
\gamma(u_{\{1\}\{2\}}(xa_2))\}=\{e, a_2\}$,
$\gamma(u_{\{1\}\{2\}}(xa_{i}))=a_{1}, \ i\in\{3,4,..., k+1\}$.
Consequently,
$$q_0(x)=|\{xa_1, xa_2,..., xa_{k+1}\}\cap K_{0}^{\ast}|=1,$$
$$q_1(x)=|\{xa_1, xa_2,..., xa_{k+1}\}\cap K_{1}^{\ast}|=k-1,$$
$$q_2(x)=|\{xa_1, xa_2,..., xa_{k+1}\}\cap K_{2}^{\ast}|=1.$$
Hence $Q(x)=(1,k-1,1)$. By Theorem \ref{thm+} there exist a
permutation of coordinates of $Q(K_{0}^{\ast})$, i.e., this
permutation $\pi_x$ of coordinates of $Q(K_{0}^{\ast})$ has the
form $\pi_x=(12)$. In Fig. 1, the partitions of $\Gamma^2$ with
respect to $K_{0}^{\ast}$ is given. The elements of the class
$K_{0}^{\ast}, K_{1}^{\ast}, K_{2}^{\ast}$ are denoted by
 \textbf{blue}, \textbf{red} and \textbf{black} respectively.
\begin{figure}
\includegraphics[width=16.0cm]{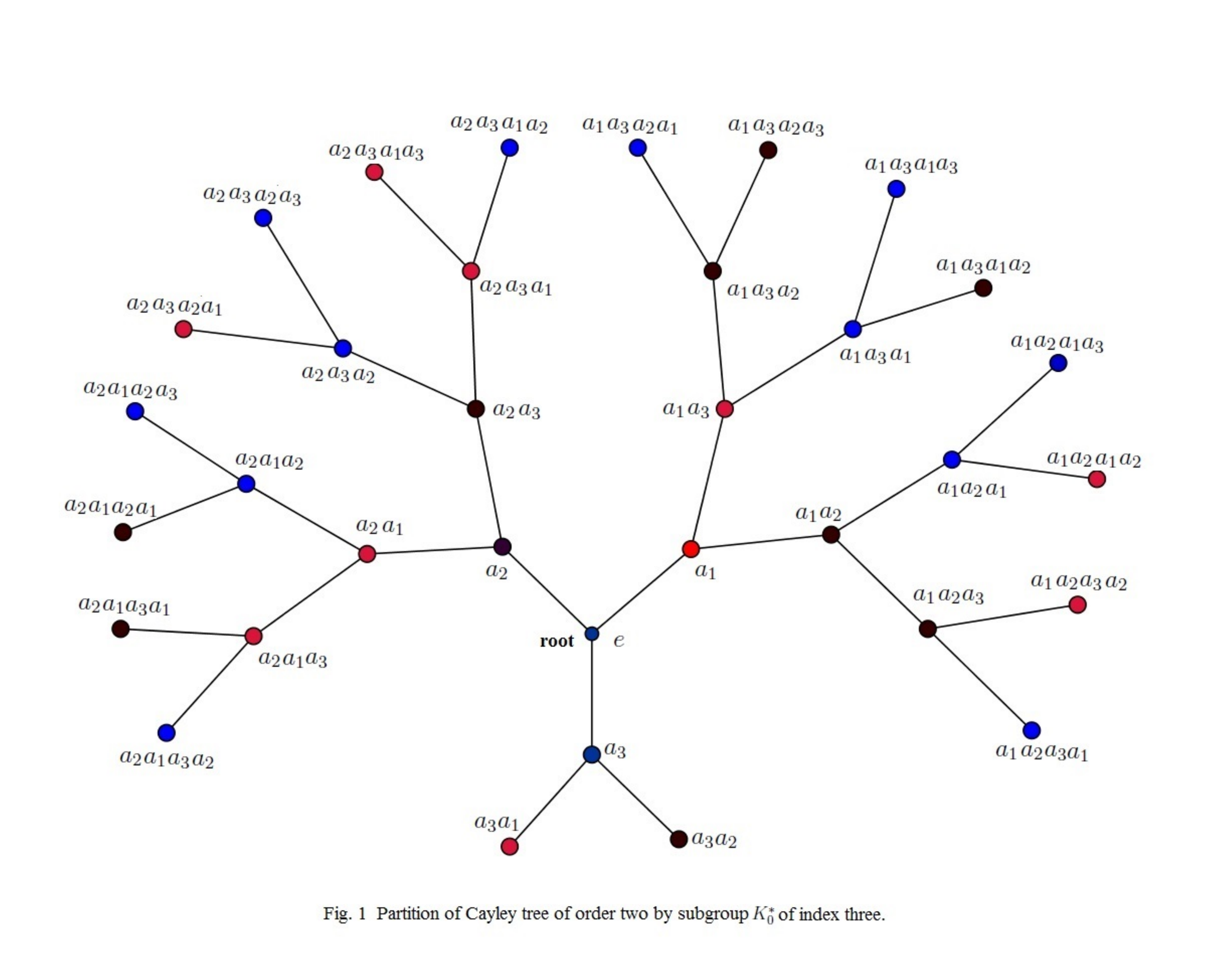}
\label{f1}
\end{figure}

Then $K_{0}^{\ast}$-weakly
periodic set of $h$ has the form
\begin{equation}\label{m5}
h_x=\left\{%
\begin{array}{ll}
    h_{1}, & {x \in K_{0}^{\ast}, \ x_{\downarrow} \in K_{0}^{\ast}}
    \\ [2mm]
    h_{2}, & {x \in K_{0}^{\ast}, \ x_{\downarrow} \in K_{1}^{\ast}}
    \\ [2mm]
    h_{3}, & {x \in K_{0}^{\ast}, \ x_{\downarrow} \in K_{2}^{\ast}}
    \\ [2mm]
    h_{4}, & {x \in K_{1}^{\ast}, \ x_{\downarrow} \in K_{0}^{\ast}}
    \\ [2mm]
    h_{5}, & {x \in K_{1}^{\ast}, \ x_{\downarrow} \in K_{1}^{\ast}}
    \\ [2mm]
    h_{6}, & {x \in K_{1}^{\ast}, \ x_{\downarrow} \in K_{2}^{\ast}}
    \\ [2mm]
    h_{7}, & {x \in K_{2}^{\ast}, \ x_{\downarrow} \in K_{0}^{\ast}}
    \\ [2mm]
    h_{8}, & {x \in K_{2}^{\ast}, \ x_{\downarrow} \in K_{1}^{\ast}}
    \\ [2mm]
    h_{9}, & {x \in K_{2}^{\ast}, \ x_{\downarrow} \in K_{2}^{\ast},}
\end{array}%
\right.
\end{equation}
where $h_{i},i=\overline{1,9}$, in view of (\ref{2.8 eq}) satisfy the following equations:
\begin{equation}\label{6}
\left\{%
\begin{array}{ll}
    h_{1}=(k-2)f(h_{1},\theta)+f(h_{4},\theta)+f(h_{7},\theta) \\ [2mm]
    h_{2}=(k-1)f(h_{1},\theta)+f(h_{7},\theta) \\ [2mm]
    h_{3}=(k-1)f(h_{1},\theta)+f(h_{4},\theta) \\ [2mm]
    h_{4}=(k-1)f(h_{5},\theta)+f(h_{8},\theta) \\ [2mm]
    h_{5}=(k-2)f(h_{5},\theta)+f(h_{2},\theta)+f(h_{8},\theta) \\ [2mm]
    h_{6}=(k-1)f(h_{5},\theta)+f(h_{2},\theta) \\ [2mm]
    h_{7}=(k-1)f(h_{9},\theta)+f(h_{6},\theta) \\ [2mm]
    h_{8}=(k-1)f(h_{9},\theta)+f(h_{3},\theta) \\[2mm]
    h_{9}=(k-2)f(h_{9},\theta)+f(h_{6},\theta)+f(h_{3},\theta).
\end{array}%
\right.\end{equation} Consider operator $W_k:R^9\rightarrow R^9$,
defined by
\begin{equation}\label{7}
\left\{%
\begin{array}{ll}
    h'_{1}=(k-2)f(h_{1},\theta)+f(h_{4},\theta)+f(h_{7},\theta)
    \\[2mm]
    h'_{2}=(k-1)f(h_{1},\theta)+f(h_{7},\theta) \\
    [2mm]
    h'_{3}=(k-1)f(h_{1},\theta)+f(h_{4},\theta) \\ [2mm]
    h'_{4}=(k-1)f(h_{5},\theta)+f(h_{8},\theta) \\ [2mm]
    h'_{5}=(k-2)f(h_{5},\theta)+f(h_{2},\theta)+f(h_{8},\theta) \\
    [2mm]
    h'_{6}=(k-1)f(h_{5},\theta)+f(h_{2},\theta) \\ [2mm]
    h'_{7}=(k-1)f(h_{9},\theta)+f(h_{6},\theta) \\ [2mm]
    h'_{8}=(k-1)f(h_{9},\theta)+f(h_{3},\theta) \\ [2mm]
    h'_{9}=(k-2)f(h_{9},\theta)+f(h_{6},\theta)+f(h_{3},\theta).
\end{array}%
\right.\end{equation}
 Note that the system of equations (\ref{6})
is the equation $h=W_k(h)$. It is obvious that the following sets
are invariant with respect to operator $W_k:$\\
 1) For case $k\geq
2$: $$I_0=\{h\in R^9: h_1=h_2=h_3=h_4=h_5=h_6= h_7=h_8=h_9\},$$ $$
I_1 =\{h\in R^9: h_1=h_2=h_4=h_5, h_3=h_6, h_7=h_8\},$$
$$I_2=\{h\in R^9: h_2=h_3; h_4=h_7; h_5=h_6=h_8=h_9\}.$$
2) For case $k=2$:
$$I_3=\{h\in R^9: h_1=h_6=h_8; h_2=h_3=h_4=h_5=h_7=h_9\},$$
$$I_4=\{h\in R^9: h_1=h_2=h_4=h_6=h_8=h_9; h_3=h_5=h_7\},$$
$$I_5=\{h\in R^9: h_1=h_3=h_5=h_6=h_7=h_8; h_2=h_4=h_9\}.$$

\begin{pro}\label{pro+} For the Ising model on $\Gamma^2$ the following assertions
hold:\\
1) All $K_{0}^{\ast}$-weakly periodic Gibbs measures on invariant
set $I_{3}$ (resp. ${{I}_{4}},{{I}_{5}}$) are translation-invariant.\\
2)For $\theta\in[0.5, 1)$ there exists two $K_{0}^{\ast}$-weakly
periodic Gibbs measures (not translation-invariant) on $I_{1}$
(resp. $I_2$). If $\theta\notin[0.5, 1)$ then any
$K_{0}^{\ast}$-weakly periodic Gibbs measure is
translation-invariant.
\end{pro}

\begin{proof} a) It suffices to show the system of equations (\ref{6})
has only root of the form $h_1=h_2=h_3=h_4=h_5=h_6= h_7=h_8=h_9$.
The proof is obvious for the set $I_0.$ Now we prove this
assertion for the invariant set ${{I}_{3}}$ (the cases
${{I}_{4}},{{I}_{5}}$ are similar).
 On the invariant set $I_3=\{h\in R^9: h_1=h_6=h_8; h_2=h_3=h_4=h_5=h_7=h_9\}$, equation (\ref{6}) can be written as
\begin{equation}\label{3.6}\left\{%
\begin{array}{ll}
    {{h}_{1}}=f({{h}_{1}},\theta )+f({{h}_{2}},\theta )\\ [2mm]
    {{h}_{2}}=f({{h}_{2}},\theta )+f({{h}_{1}},\theta ).
\end{array}%
\right.\end{equation}
 Right hand sides of the system are the same, and that's why one gets $h_{1}=h_{7}$.
Hence, all $K_{0}^{\ast}$- weakly periodic Gibbs measures on
invariant set $I_3$ are translation-invariant.

 b) Now we consider equation (\ref{6}) on the invariant sets:
${{I}_{1}},{{I}_{2}}$. In this case, it is sufficient to analyze
the following equation:
\begin{equation}\label{3.7}\left\{%
\begin{array}{ll}
 {{h}_{1}}=f({{h}_{1}},\theta)+f({{h}_{7}},\theta)\\ [2mm]
 {{h}_{3}}=2f({{h}_{1}},\theta)\\ [2mm]
 {{h}_{7}}=f({{h}_{9}},\theta)+f({{h}_{3}},\theta)\\ [2mm]
 {{h}_{9}}=2f({{h}_{3}},\theta) \\
\end{array} \right. \end{equation}
which is obtained by restricting the operator $W_2$ to $I_1$ (the
case ${{I}_{2}}$ is similar). Using the fact that
$$f(h,\theta )=\frac{1}{2}\ln \frac{(1+\theta ){{e}^{2h}}+1-\theta
}{(1-\theta ){{e}^{2h}}+1+\theta }$$ and introducing the notations
${{z}_{i}}={{e}^{2{{h}_{i}}}},\ i=\{1,3,7,9\}$ and
$a=\frac{1-\theta }{1+\theta}$ we obtain the following system of
equations instead of (\ref{3.7}):
$$\left\{%
\begin{array}{ll}
 {{z}_{1}}=g({{z}_{1}})g({{z}_{7}}) \\ [2mm]
 {{z}_{3}}={{g}^{2}}({{z}_{1}}) \\ [2mm]
 {{z}_{7}}=g({{z}_{9}})g({{z}_{3}}) \\ [2mm]
 {{z}_{9}}={{g}^{2}}({{z}_{3}}),
\end{array}\right.$$
where $g(z)=\frac{z+a}{az+1}$. The last system can be rewritten as
the following form:
\begin{equation}\label{3.8}{{g}^{-1}}\left( \frac{{{g}^{-1}}
(\sqrt{{{z}_{3}}})}{\sqrt{{{z}_{3}}}} \right)= g\left(
{{g}^{2}}\left( {{z}_{3}} \right) \right)g\left( {{z}_{3}} \right)
\end{equation}
Using simple analysis one gets
$$\frac{\left( {{a}^{2}}{{x}^{2}}+\left( 1-a \right)\cdot x-a
\right)}{-a{{x}^{2}}+\left( 1-a \right)\cdot
x+{{a}^{2}}}=\frac{\left( \left( {{a}^{2}}-a+1 \right)\cdot
{{x}^{6}}+\left( {{a}^{3}}-{{a}^{2}}+3a \right){{x}^{4}}+\left(
2{{a}^{2}}+a \right){{x}^{2}}+{{a}^{2}} \right)}{\left(
{{a}^{2}}{{x}^{6}}+\left( 2{{a}^{2}}+a \right)\cdot
{{x}^{4}}+\left( {{a}^{3}}-{{a}^{2}}+3a
\right){{x}^{2}}+{{a}^{2}}-a+1 \right)},$$ where
$\sqrt{{{z}_{3}}}=x$.
  From simple calculation, we have
$$P(x):=\left(x-1 \right)\left( x+1 \right)\left( a{{x}^{2}}
+x(a-1)+a \right)\cdot Q(x)=0,$$ where
$$Q(x):=\left({{a}^{3}}+{{a}^{2}}-a+1 \right){{x}^{4}}+\left(
a-{{a}^{3}} \right){{x}^{3}}+\left( 3{{a}^{3}}-{{a}^{2}}+a+1
\right){{x}^{2}}+\left( a-{{a}^{3}}
\right)x+{{a}^{3}}+{{a}^{2}}-a+1.$$ At first, we show that there
is not any positive root of the polynomial $Q(x)$.

Let $x\in \left[ 0,1 \right]$. If $a-a^3\geq 0$ then $Q(x)>0.$ By
this occasion, it's sufficient to check the case $a-a^3\leq 0$:
$$Q(x)=\left({{a}^{3}}+{{a}^{2}}-a+1 \right)({{x}^{4}+1})+\left(
a-{{a}^{3}} \right)({{x}^{3}+x})+\left(3{{a}^{3}}-{{a}^{2}}+a+1
\right){{x}^{2}}\geq $$
$$\left({{a}^{3}}+{{a}^{2}}-a+1 \right)({{x}^{4}+1})+2\left(
a-{{a}^{3}} \right)+\left(3{{a}^{3}}-{{a}^{2}}+a+1
\right){{x}^{2}}$$ From ${{a}^{3}}+{{a}^{2}}-a+1>0$ and
$3{{a}^{3}}-{{a}^{2}}+a+1>0$ we get
$$Q(x)\geq\left( {{a}^{3}}+{{a}^{2}}-a+1 \right)\left( {{x}^{4}}+1
\right)+2\left( a-{{a}^{3}}
\right)={{a}^{3}}+{{a}^{2}}-a+1+2\left( a-{{a}^{3}}
\right)=-{{a}^{3}}+{{a}^{2}}+a+1>0$$

Now we consider the case $x\in (1,\infty)$. Then
$$Q(x)=\left({{a}^{3}}+{{a}^{2}}-a+1 \right){{x}^{4}}+\left( a-{{a}^{3}}
 \right){{x}^{3}}+\left( a-{{a}^{3}} \right)x\geq$$
 $$ \left( {{a}^{3}}+{{a}^{2}}-a+1 \right){{x}^{4}}+
 \left( a-{{a}^{3}} \right){{x}^{4}}+\left( a-{{a}^{3}}
 \right)x=$$
 $$\left( {{a}^{2}}+1 \right){{x}^{4}}+\left( a-{{a}^{3}}
 \right)x>
  \left( {{a}^{2}}+1 \right)x+\left( a-{{a}^{3}} \right)x\geq
  \left( -{{a}^{3}}+{{a}^{2}}+a+1\right)>0. $$
 Hence, there is not any positive solution of $Q(x)$.
Instead of studying positive roots of $P(x)$, it's enough to check
positive solutions of the following equation:
 $\left( x-1 \right)\left(x+1 \right)\left(a{{x}^{2}}+x(a-1)+a \right)=0$.
 If $x=1$ then solutions of the system (\ref{3.8}) coincide with
 translation-invariant solutions. It's easy to check that if $a\in(0,\frac{1}{3}]$ (resp. $\theta\in[0.5, 1)$) then
 the quadratic equation $a{{x}^{2}}+x(a-1)+a=0$ has two positive solutions, i.e., $x_{1,2}=\frac{1-a\pm\sqrt{1-3a^2-2a}}{2a}$.
If $a\notin(0,\frac{1}{3}]$ then this quadratic equation has not
any positive solution. This completes the proof.

\end{proof}

\end{document}